\documentclass[11pt]{article}

\usepackage{amsmath}
\usepackage{lineno}
\usepackage{amsthm}
  \usepackage{paralist}
  \usepackage{graphics} 
  \usepackage{epsfig} 
\usepackage{graphicx}
\usepackage{caption}
\usepackage{subcaption}
\usepackage{epstopdf}
 \usepackage[colorlinks=true]{hyperref}
 \usepackage{multirow}
\input{amssym.tex}

\newtheorem{theorem}{Theorem}[section]

\newtheorem{lemma}[theorem]{Lemma}
\newtheorem{proposition}{Proposition}

 \numberwithin{equation}{section}
\newtheorem{remark}{Remark}

\newcommand{\keywords}

\def\bc{\begin{center}}       \def\ec{\end{center}}
\def\ba{\begin{array}}        \def\ea{\end{array}}
\def\be{\begin{equation}}     \def\ee{\end{equation}}
\def\bea{\begin{eqnarray}}    \def\eea{\end{eqnarray}}
\def\beaa{\begin{eqnarray*}}  \def\eeaa{\end{eqnarray*}}

\def\mathbb{\Bbb}

\begin{document}

\title{\bf Global well-posedness of advective Lotka–Volterra competition systems with nonlinear diffusion\footnote{published in Proc. Roy. Soc. Edinburgh Sect. A, 2019,
https://doi.org/10.1017/prm.2019.10}}
\author{Qi Wang\thanks{{\tt qwang@swufe.edu.cn}, corresponding author.  Department of Mathematics, Southwestern University of Finance and Economics, Chengdu, Sichuan 611130, China.  QW is supported by NSF-China (Grant 11501460) and the Fundamental Research Funds for the Central Universities, China (JBK1805001)},
Jingyue Yang \thanks{{\tt yjy@2011.swufe.edu.cn}.  Department of Mathematics, Southwestern University of Finance and Economics, Chengdu, Sichuan 611130, China},
Feng Yu  \thanks{{\tt yfeng@knights.ucf.edu}.  Department of Mathematics,  University of Central Florida, Orlando, Florida 32816, USA}
}

\date{}
\maketitle

\abstract
This paper investigates the global well--posedness of a class of reaction--advection--diffusion models with nonlinear diffusion and Lotka--Volterra dynamics.  We prove the existence and uniform boundedness of the global--in--time solutions to the fully parabolic systems under certain growth conditions on the diffusion and sensitivity functions.  Global existence and uniform boundedness of the corresponding parabolic--elliptic system are also obtained.  Our results suggest that attraction (positive taxis) inhibits blowups in Lotka--Volterra competition systems.


\section{Introduction}\label{section1}
This paper is concerned with the global existence and boundedness of $(u,v)=(u(x,t),v(x,t))$ to reaction--advection--diffusion systems of the following form
\begin{equation}\label{11}
\left\{
\begin{array}{ll}
u_t=\nabla \cdot (D_1(u) \nabla u+\chi \phi(u) \nabla v)+(a_1-b_1u^{\alpha}-c_1v)u,&x \in \Omega,t>0, \\
v_t=D_2\Delta v+(a_2-b_2u-c_2v)v,&x \in \Omega,t>0, \\
\frac{\partial u}{\partial \textbf{n}}=\frac{\partial v}{\partial \textbf{n}}=0,&x\in\partial \Omega,t>0,\\
u(x,0)=u_0(x)\geq 0,v(x,0)=v_0(x)\geq 0, &x\in \Omega.
\end{array}
\right.
\end{equation}
Here $\Omega$ is a bounded domain in $\mathbb{R}^N, N\geq 1$ with a piece-wise smooth boundary $\partial \Omega$ endowed with unit outer normal $\textbf{n}$.  $a_i$, $b_i$, $c_i$, $i=1,2$, $D_2$ and $\chi$ are positive constants, while $D_1$ and $\phi$ are $C^2$--smooth functions of $u$.  Moreover, we assume there exist some positive constants $M_i$, $m_i$, $i=1,2$, such that
\begin{equation}\label{12}
D_1(u)\geq M_1(1+u)^{m_1}, \forall u\geq 0,
\end{equation}
and
\begin{equation}\label{13}
0\leq \phi(u)\leq M_2u^{m_2}, \forall u\geq 0.
\end{equation}

System (\ref{11}) can be used to model the evolution of population distributions of two competing species subject to Lotka--Volterra dynamics.  Consider two species with population densities being $u(x,t)$ and $v(x,t)$ at location $x\in\Omega$ and time $t\geq0$.  Diffusion describes the random dispersal of the species as an anti--crowding mechanism and it is taken to be spatially local and against the direction of population gradient of the focal species.  Moreover such anti--crowding motion changes with respect to the variation of the population density, and therefore we assume that $D_1$ is a function of $u$, while $D_2$ is chosen to be a positive constant for the simplicity of our analysis.  The advection $\chi\phi(u)\nabla v$, or the cross--diffusion, accounts for the directed dispersal due to the population pressure from competing species $v$, and it is along with the direction of population gradient $\nabla v$.  In (\ref{11}) the function $\phi(u)$ interprets variation of the advection intensity with respect to population density $u$.   The population kinetics are assumed to be of Lotka--Volterra type.

The initial step to understand the spatial--temporal dynamics of (\ref{11}) is to study its global well--posedness.  It is the goal of this paper to study the effects of growth rates $m_i$ and decay rate $\alpha$, although far from being well understood, on the global existence and uniform boundedness of this system.  Our first main result reads as follows.
\begin{theorem}\label{theorem11}
Suppose that $\Omega\subset \mathbb{R}^N, N\geq2$, is a bounded domain.  Assume that the smooth functions $D_1(u)$ and $\phi(u)$ satisfy (\ref{12}) and (\ref{13}) respectively with
\begin{align}\label{14}
m_2-m_1<
    \begin{cases}
       \frac{2}{N},&\text{~if~} 0<\alpha < 1,\\
       \frac{3N+2}{N(N+2)},&\text{~if~} \alpha\geq 1,\\
    \end{cases}
\end{align}
 then for any nonnegative $(u_0,v_0)\in C(\bar{\Omega}) \times W^{1,\infty}(\Omega)$, there exists at least one couple $(u,v)$ of nonnegative bounded functions each belonging to $C^0(\bar{\Omega}\times[0,\infty))\cap C^{2,1}(\bar{\Omega}\times(0,\infty))$ which solves (\ref{11}) classically.  Moreover if $(u_0,v_0)\in W^{k,p}(\Omega)\times W^{k,p}(\Omega)$ for some $k>1$ and $p>N$, the bounded solution above is unique.
\end{theorem}
By a different approach we are also able to prove the following result under a condition different from (\ref{14}).

\begin{theorem}\label{theorem12}
Suppose that all conditions in Theorem \ref{theorem11} hold except that (\ref{14}) is replaced by
  \begin{align}\label{15}
2m_2-m_1<
    \begin{cases}
      \max\{\alpha, m_1\}+\frac{2}{N},&\text{~if~} 0<\alpha < 1,\\
      \max\{\alpha, m_1\}+\frac{4}{N+2},&\text{~if~} \alpha\geq 1,\\
    \end{cases}
  \end{align}
then all the conclusions in Theorem \ref{theorem11} hold, i.e., the solution to (\ref{11}) is unique, global and uniformly bounded in time.
\end{theorem}

In the absence of advection, for example when $D_1(u)\equiv D_1$, $\chi=0$ and $\alpha=1$, (\ref{11}) reduces to the following diffusive Lotka--Volterra competition model
\begin{equation}\label{16}
\left\{
\begin{array}{ll}
u_t=D_1\Delta u+(a_1-b_1u-c_1v)u,&x \in \Omega,t>0, \\
v_t=D_2\Delta v+(a_2-b_2u-c_2v)v,&x \in \Omega,t>0, \\
\frac{\partial u}{\partial \textbf{n}}=\frac{\partial v}{\partial \textbf{n}}=0,&x\in\partial \Omega,t>0,\\
u(x,0)=u_0(x)\geq 0,v(x,0)=v_0(x)\geq 0, &x\in \Omega.
\end{array}
\right.
\end{equation}
Thanks to the standard parabolic maximum principles, it is quite obvious that the solution $(u,v)$ to (\ref{16}) exists globally and is uniformly bounded \cite{CHS,CS}.  It is also known that the positive homogeneous solutions $(\bar u,\bar v)$ is the global (exponential) attractor of (\ref{16}) in weak competition case $\frac{b_1}{b_2}>\frac{a_1}{a_2}>\frac{c_1}{c_2}$ \cite{CHS,DR}, and (\ref{16}) does not admit nonconstant stable steady states when $\Omega$ is convex \cite{KiW} or one of the diffusion rates $D_i$ is large \cite{DR,LN}.  On the other hand, the system admits nonconstant positive steady states when $\Omega$ is non--convex (e.g. of dumb--bell shaped) in the strong competition case $\frac{b_1}{b_2}<\frac{a_1}{a_2}<\frac{c_1}{c_2}$, with properly chosen (small) diffusion rates \cite{MM,MEF,MK,MNTT}.  See \cite{LN,LN2,WGY} for further discussions on (\ref{16}).

Though it is not entirely unrealistic to assume that mutually interacting species disperse over the habitat purely randomly, from the viewpoint of mathematical modeling, it is interesting to incorporate advection or cross--diffusion into system (\ref{16}), which accounts for the dispersal pressure due to population gradient of the intra-- and/or inter--species.  On the other hand, one of the most interesting phenomena in ecological evolution is the well--observed segregation of competing species, i.e., some regions of the habitat are dominated by one species and the rest by the other; however, in most cases system (\ref{16}) inhibits the formation of nontrivial patterns such as boundary spikes, transition layers etc., which can be used to model the aforementioned segregation.  For this purpose, the following model with advection was proposed and studied in \cite{WGY}
\begin{equation}\label{17}
\left\{
\begin{array}{ll}
u_t=\nabla \cdot (D_1 \nabla u+\chi u \nabla v)+(a_1-b_1u-c_1v)u,&x \in \Omega,t>0, \\
v_t=D_2\Delta v+(a_2-b_2u-c_2v)v,&x \in \Omega,t>0, \\
\frac{\partial u}{\partial \textbf{n}}=\frac{\partial v}{\partial \textbf{n}}=0,&x\in\partial \Omega,t>0,\\
u(x,0)=u_0(x)\geq 0,v(x,0)=v_0(x)\geq 0, &x\in \Omega,
\end{array}
\right.
\end{equation}
where all the parameters are positive constants.  Global existence and boundedness of the fully parabolic system are obtained in \cite{WGY} when $\Omega$ is one--dimensional and of its parabolic--elliptic counterpart when $\Omega$ is multi--dimensional and $\frac{\chi}{D_2}$ is small.  Steady state bifurcation is performed to establish the existence and stability of its nonconstant stationary solutions.  Moreover, it is shown that (\ref{17}) admits transition--layer steady states when $\chi$ and $1/D_2$ are sufficiently large.  These nonconstant steady states can be used to model the aforementioned segregation phenomenon.  Recently it is proved in \cite{SSW} that extinction through competition does not occur in (\ref{17}) out of small initial data in the weak competition case.  Global existence and nonconstant steady states of (\ref{11}) with sublinear sensitivity are obtained in \cite{WZ} when $\Omega$ is a multi--dimensional bounded domain.

In this work, we extend (\ref{17}) to the more realistic and general model (\ref{11}) with nonconstant diffusion by assuming that the random dispersal rate of species depends non-linearly on the population density of the focal species $u$.  Moreover, the density--dependent sensitivity means that the advective velocity of species $u$ varies with different population density.  By nonlinear diffusion and sensitivity, we are able to use (\ref{11}) to describe population--induced dispersal in ecological applications.  Here for the simplicity of our analysis and to focus the interplay between $m_i$ and $\alpha$ on our results, we always assume that $D_2$ is a positive constant.

We would like to mention that (\ref{11}) serves as a prototype for reaction--diffusion systems with cross--diffusion modeling population pressures created by the competitions.  For example, Shigesada, Kawasaki and Teramoto \cite{SKT} proposed the following system in 1979 to model the segregation phenomenon of two competing species
\begin{equation}\label{18}
\left\{
\begin{array}{ll}
u_t=\Delta[(d_1+\rho_{11}u+\rho_{12}v)u]+(a_1-b_1u-c_1v)u, &x \in \Omega,~t>0,     \\
v_t=\Delta[(d_2+\rho_{21}u+\rho_{22}v)v]+(a_2-b_2u-c_2v)v,& x \in \Omega,~t>0, \\
\frac{\partial u}{\partial \textbf{n}}=\frac{\partial v}{\partial \textbf{n}}=0,& x \in \partial \Omega,~t>0,\\
u(x,0)=u_0(x) \geq 0,~ v(x,0)=v_0(x) \geq 0,& x\in \Omega,
\end{array}
\right.
\end{equation}
which takes into consideration both \textit{self--diffusions} $\rho_{11},\rho_{22}$ and \textit{cross--diffusions} $\rho_{12},\rho_{21}$.  (\ref{18}) has received adequate attention over the past few decades since its appearance and a great deal of effort has been devoted to studying its global well--posedness \cite{CLY,CLY2,LeNN,LNW,Shim,Tuoc,TuP,Yamada} and positive steady states \cite{KW,Kuto,KT,LN,LN2,LNY,LNY2,NWX,Wq,WX}.  To compare (\ref{18}) with (\ref{11}), we assume that $\rho_{21}=\rho_{22}=0$ and rewrite it into the following form
\begin{equation}\label{19}
\left\{
\begin{array}{ll}
u_t=\nabla \cdot [(d_1+2\rho_{11}u+\rho_{12}v) \nabla u+\rho_{12}u\nabla v]+(a_1-b_1u-c_1v)u, &x \in \Omega,~t>0,     \\
v_t= d_2\Delta v+(a_2-b_2u-c_2v)v,& x \in \Omega,~t>0, \\
\frac{\partial u}{\partial \textbf{n}}=\frac{\partial v}{\partial \textbf{n}}=0,& x \in \partial \Omega,~t>0,\\
u(x,0)=u_0(x) \geq 0,~ v(x,0)=v_0(x) \geq 0,& x\in \Omega.
\end{array}
\right.
\end{equation}
It is proved in \cite{LNW} that when space dimension $N=2$, if $u_0,v_0\in W^{k,p}$ for some $k>N$, then (\ref{19}) has a unique global solution which solves the system classically.  As can be easily seen, this global existence result can be rediscovered by both Theorem \ref{11} and Theorem \ref{21} since (\ref{19}) is a special case of (\ref{11}) with $m_2=m_1=\alpha=1$, for which (\ref{14}) and (\ref{15}) obviously hold.  Moreover our results show that the global solutions are uniformly bounded in time which was not available in \cite{LNW}.

Another example is the following model proposed in \cite{CCL,CCLX} to study the dispersal strategies leading to ideal free distribution of populations in evolutionary ecology
\begin{equation}\label{110}
\left\{
\begin{array}{ll}
u_t=\nabla \cdot (d_1\nabla u-\chi u \nabla (m-u-v))+(m-u-v)u,&x \in \Omega,t>0, \\
v_t=d_2\Delta v+r(m-u-v)v,&x \in \Omega,t>0, \\
\frac{\partial u}{\partial \textbf{n}}=\frac{\partial v}{\partial \textbf{n}}=0,&x\in\partial \Omega,t>0,\\
u(x,0)=u_0(x)\geq 0,v(x,0)=v_0(x)\geq 0, &x\in \Omega,
\end{array}
\right.
\end{equation}
where $d_1$, $d_2$ and $\chi$ are positive constants.  $m=m(x)\in C^{2+\gamma}(\bar \Omega)$ and $m(x)>0$ in $\Omega$.  Lou \emph{et al.} \cite{LTW} studied the bounded classical global solutions to the following system over multi--dimensional domain $N\geq1$. We note that (\ref{110}) can be rewritten as
\[
\left\{
\begin{array}{ll}
u_t=\nabla \cdot ((d_1+\chi u)\nabla u+\chi u \nabla v-\chi u\nabla m)+(m-u-v)u,&x \in \Omega,t>0, \\
v_t=d_2\Delta v+r(m-u-v)v,&x \in \Omega,t>0, \\
\frac{\partial u}{\partial \textbf{n}}=\frac{\partial v}{\partial \textbf{n}}=0,&x\in\partial \Omega,t>0,\\
u(x,0)=u_0(x)\geq 0,v(x,0)=v_0(x)\geq 0, &x\in \Omega,
\end{array}
\right.
\]
hence it is a special case of (\ref{11}) with $m_1=m_2=\alpha=1$ and the global well--posedness follows from Theorem \ref{theorem11} or Theorem \ref{theorem12}.  We refer to \cite{CC,CJun2,Cosner,DLMT,HNP,Jun,Le,LeN,LeN2,SK} and the references therein for works on global existence of cross--diffusion systems.

We would like to mention that (\ref{11}) is very similar to the nonlinear diffusion Keller--Segel models of chemotaxis, which describes the directed movements of cellular organisms in response to chemical stimulus.  In particular, the chemotaxis is attractive if the cells move towards high concentration of the chemical (e.g., sugar, nutrition) and chemotaxis is repulsive if the cells move against the chemical concentration (e.g., poison, hazardous materials).  It is also necessary to point out that the logistic growth in Lotka--Volterra dynamics, which inhibits solutions from blowing up within a finite or infinite time for purely diffusive models, might not be sufficient to prevent blowups when advection or chemotaxis is present.  For example, Le and Nguyen \cite{LN} gave an example of finite--time blowup solutions to a cross--diffusion system subject to Lotka--Volterra dynamics.  See \cite{Lan,Winkler3,Winkler4} for counterexamples for chemotaxis models with logistic growth, and \cite{BDD,ISY,TW2,WMZ,Wyf,ZL,ZMH} for the works on chemotaxis models with nonlinear diffusions.

The rest part of this paper is organized as follows.  In Section \ref{section2}, we present the existence and important some important properties of the local in time solution to (\ref{11}).  In Section \ref{section3}, we establish several \emph{a priori} estimates which are essential for the proof of Theorem \ref{theorem11} and Theorem \ref{theorem12}.  Finally in Section \ref{section4}, for parabolic--elliptic system of (\ref{11}) with repulsion, we prove its global existence and boundedness in Theorem \ref{theorem41} and Theorem \ref{theorem42} under milder conditions on $m_i$ and $\alpha$ than (\ref{14}) and (\ref{15}); moreover for parabolic--elliptic of system (\ref{11}) with attraction (i.e., change $\chi$ to $-\chi$), we prove in Theorem \ref{theorem43} that the solutions are global and bounded for as long as one of $m_1$, $m_2$ and $\alpha$ is nonnegative.  Our results indicate that, unlike Keller--Segel models in which chemo--repulsion is a smoothing process, population repulsion destabilizes the spatially homogeneous solution of Lotka--Volterra competition systems (e.g., see Proposition 1 in \cite{WGY}).

\section{Local existence and preliminary results}\label{section2}
The mathematical analysis of global well--posedness of (\ref{11}) is delicate since the maximum principle does not apply for the $u$ equation.  However, the local well--posedness follows easily from the fundamental theory developed by Amann \cite{Am} and the standard parabolic regularity theory.
\begin{proposition}\label{proposition21}
Let $\Omega$ be a bounded domain in $\mathbb R^N$, $N\geq1$.
Let $a_i, b_i, c_i, \alpha, D_2$ be positive and suppose that $D_1(u)$ and $\phi(u)$ are $C^2$ smooth functions and they satisfy (\ref{12}) and (\ref{13}) for positive constants $m_i$ and $M_i$, $i=1,2$.  Assume that for some $p>N$ and $k>1$, $(u_0,v_0)$ belongs to $(W^{k,p}(\Omega))^2$ and $u_0,v_0\geq,\not\equiv 0 $ in $\bar{\Omega}$.  Then there exist $T_{\max}\in(0,\infty]$ and a unique couple $(u,v)$ of nonnegative functions from $C^0(\bar{\Omega}\times[0,T_{\max}))\cap C^{2,1}(\bar{\Omega}\times(0,T_{\max}))$ solving (\ref{11}) classically in $\Omega\times(0,T_{\max})$.  Moreover $u(x,t)\geq 0$ and $v(x,t)\geq0$ in $\Omega\times(0,T_{\max})$ and the following dichotomy holds:
\begin{equation}\label{21}
  \text{either \quad} T_{\max}=\infty \text{\qquad or\qquad} T_{\max}<\infty \text{~with~} \limsup_{t\nearrow T_{\max}^-} \|u(\cdot,t)\|_{L^{\infty}(\Omega)}=\infty.
\end{equation}
\end{proposition}

Next we collect some properties of the local solution.
\begin{lemma}\label{lemma22}
Let $(u,v)$ be a nonnegative classical solution of (\ref{11}) in $\Omega\times(0,T_{\max})$.  Then the following statements hold true:

(i)  there exists a positive constant $C$ such that
\begin{equation}\label{22}
\int_\Omega u(x,t)dx\leq C, \forall t\in(0,T_{\max})
\end{equation}
and
\begin{equation}\label{23}
0< v(x,t) \leq \max\Big\{\frac{a_2}{c_2},\Vert v_0\Vert_{L^\infty(\Omega)}\Big\}, \forall (x,t)\in\Omega \times (0,T_{\max});
\end{equation}

(ii) for each $s\in[1,\frac{N}{N-1})$, there exists $C_s>0$ such that
\begin{equation}\label{24}
\Vert v (\cdot,t) \Vert_{W^{1,s}(\Omega)} \leq C_s, \forall t\in(0,T_{\max});
\end{equation}
moreover if $u\in L^p(\Omega)$ for some $p\in[1,\infty)$, there exists a positive constant $C$ dependent on $\Vert v_0 \Vert _{L^{q}(\Omega)}$ and $\vert \Omega\vert$ such that
\begin{equation}\label{25}
\Vert v(\cdot,t) \Vert_{W^{1,q}(\Omega)} \le C\Big(1+\sup_{s\in(0,t)}  \Vert u(\cdot,s)\Vert_{L^p(\Omega)}\Big), \forall t\geq0,
\end{equation}
where $q\in[1,\frac{Np}{N-p})$ if $p\in [1,N)$, $q\in [1,\infty)$ if $p=N$ and $q=\infty$ if $p>N$.
\end{lemma}

\begin{proof}
First of all, we can derive the nonnegativity of $u(x,t)$ and (\ref{23}) by the standard parabolic maximum principles and Hopf's boundary point lemma.  To show (\ref{22}), we integrate the $u$--equation in (\ref{11}) over $\Omega$ to get
\begin{align}\label{26}
    \frac{d}{dt}\int_{\Omega}u=a_1 \int_{\Omega}u-b_1\int_{\Omega}u^{\alpha+1}-c_1\int_{\Omega}uv \leq a_1 \int_{\Omega}u-b_1\int_{\Omega}u^{\alpha+1}.
  \end{align}
After applying the Young's inequality $(a_1+1) \int_{\Omega}u\leq b_1\int_{\Omega}u^{\alpha+1}+C_\Omega$ for some positive constant $C_\Omega$, we obtain from (\ref{26}) that
\[\frac{d}{dt}\int_{\Omega}u+\int_{\Omega}u\leq C\]
and solving this differential inequality by Gr\"{o}nwall's lemma leads us to (\ref{22}).

To prove \emph{(ii)}, we observe that (\ref{24}) is a special case of (\ref{25}) with $p=1$ and therefore we shall only verify the latter.  To this end, we write the following abstract formula of $v$
\begin{equation}\label{27}
v(\cdot,t)=e^{D_2 (\Delta-1)t}v_0+\int_0^t e^{D_2 (\Delta-1)(t-s)} \big(D_2v(\cdot,s)+g(u(\cdot,s),v(\cdot,s)) \big)ds,
\end{equation}
where $g(u,v)=(a_2-b_2u-c_2v)v$.  Thanks to the $L^p$--$L^q$ estimates between semigroups $\{e^{t\Delta}\}_{t\geq0}$ (e.g., Lemma 1.3 of \cite{Winkler}), we can find positive constants $C_{21}$, $C_{22}$ and $C_{23}$ such that
\begin{align}\label{28}
&\Vert v(\cdot,t) \Vert _{W^{1,q}}\nonumber \\
=&\Big \Vert e^{D_2 (\Delta-1)t}v_0+\int_0^t e^{D_2 (\Delta-1)(t-s)} \big(D_2v(\cdot,s)+g(u(\cdot,s),v(\cdot,s)) \big)ds\Big\Vert _{W^{1,q}} \nonumber\\
\leq& C_{21} \Vert v_0\Vert_{L^p}+C_{21}\int_0^t e^{-D_2\nu(t-s)}(1+(t-s)^{-\frac{1}{2}-\frac{N}{2}(\frac{1}{p}-\frac{1}{q})})(\Vert u(\cdot,t)\Vert_{L^p}+1) ds  \nonumber\\
\leq& C_{22} +C_{23}\int_0^t e^{-D_2\nu(t-s)}(1+(t-s)^{-\frac{1}{2}-\frac{N}{2}(\frac{1}{p}-\frac{1}{q})}) \Vert u(\cdot,s)\Vert_{L^p} ds \nonumber \\
\leq& C_{22}+C_{23} \Big(\int_0^t e^{-D_2\nu(t-s)}(1+(t-s)^{-\frac{1}{2}-\frac{N}{2}(\frac{1}{p}-\frac{1}{q})})ds\Big)\sup_{s\in(0,t)}  \Vert u(\cdot,s)\Vert_{L^p},
\end{align}
where $\nu$ is the first Neumann eigenvalue of $-\Delta$.  On the other hand, under the conditions on $q$ behind (\ref{25}) we have
\[ \sup_{t\in(0,\infty)}\int_0^t e^{-D_2\nu(t-s)}(1+(t-s)^{-\frac{1}{2}-\frac{N}{2}(\frac{1}{p}-\frac{1}{q})})  ds<\infty,\]
and therefore (\ref{25}) follows from (\ref{28}).
\end{proof}
According to (\ref{24}) in Lemma \ref{lemma22}, $\Vert\nabla v(\cdot,t)\Vert_{L^s}$ is bounded for $s\in[1,\frac{N}{N-1})$.  Therefore for $N=1$, one has the boundedness of $\Vert\nabla v(\cdot,t)\Vert_{L^s}$ for each fixed $s\in[1,\infty)$.  By the standard Moser--Alikakos iteration we can easily prove the global existence and boundedness in Theorem \ref{theorem11} and Theorem \ref{theorem12}.  Therefore, in the sequel we shall focus on $N\geq2$ for which one has the boundedness of $\Vert\nabla v(\cdot,t)\Vert_{L^s}$ for each fixed $s\in[1,\frac{N}{N-1})$.  We want to point out that $\frac{N}{N-1}\leq2$ if $N\geq2$ and our next result indicates that $s=2$ can be achieved if $\alpha\geq1$.
\begin{lemma}
Suppose that $\alpha\geq 1$, then there exists a positive constant $C$ such that
\begin{equation} \label{29}
\Vert \nabla v(\cdot,t) \Vert_{L^2(\Omega)} \leq C, \forall t\in(0,T_{\max}).
\end{equation}
\end{lemma}
\begin{proof}
 Testing the $v$-equation in (\ref{11}) by $\Delta v$ and then integrating it over $\Omega$ by parts, we have
\begin{align}\label{210}
    \frac{1}{2}\frac{d}{dt}\int_{\Omega}\vert \nabla v \vert^2 =&\int_{\Omega}\nabla v \cdot \nabla v_t \nonumber \\
    =&\int_{\Omega}\nabla v\cdot\nabla [D_2\Delta v+(a_2-b_2u-c_2v)v] \nonumber \\
    =&  -D_2\int_{\Omega}\vert \Delta v \vert^2 + a_2\int_{\Omega}\vert \nabla v \vert^2 + \int_{\Omega} b_2 uv\Delta v -2c_2\int_\Omega v|\nabla v|^2 \nonumber \\
    \leq&   -D_2\int_{\Omega}\vert \Delta v \vert^2 + a_2\int_{\Omega}\vert \nabla v \vert^2  + \frac{b^2_2}{2D_2}\int_{\Omega} u^2v^2 +\frac{D_2}{2}\int_{\Omega}|\Delta v|^2 \nonumber \\
     \leq& -\frac{D_2}{2}\int_{\Omega}\vert \Delta v \vert^2+a_2\int_{\Omega}\vert \nabla v \vert^2+\mu\int_{\Omega} u^2,
\end{align}
where $\mu:=\frac{b^2_2\|v\|^2_{L^{\infty}(\Omega)}}{2D_2}$ and $C_{24}$ is a positive constant.  By Sobolev interpolation inequality and in light of the boundedness of $\Vert v\Vert_{L^\infty(\Omega)}$, we obtain that for positive constants $C_{25}$ and $C_{26}$
  \begin{align*}
    \Big(a_2+\frac{1}{2}\Big)\int_{\Omega}\vert \nabla v\vert ^2\leq \frac{D_2}{2}\int_{\Omega}\vert \Delta v\vert^2+C_{25}\int_{\Omega} v^2\leq \frac{D_2}{2}\int_{\Omega}\vert \Delta v\vert^2+C_{26}.
  \end{align*}
Multiplying (\ref{26}) by $\frac{2\mu}{b_1}$ and then adding it to (\ref{210}), we have
  \begin{align*}
    &\frac{d}{dt}\Big(\frac{2\mu}{b_1}\int_{\Omega}u +\frac{1}{2}\int_{\Omega}\vert \nabla v \vert^2 \Big)
    +\Big(\frac{2\mu}{b_1}\int_{\Omega}u +\frac{1}{2}\int_{\Omega}\vert \nabla v \vert^2 \Big)      \\
    \leq &\Big(\frac{2a_1\mu}{b_1} \int_{\Omega}u-\mu\int_{\Omega}u^{\alpha+1}\Big)
    +\mu\Big(\int_{\Omega}u^{2} -\int_{\Omega}u^{\alpha+1} \Big)+C_{27}
     \leq C_{28},
  \end{align*}
where $C_{27}$ and $C_{28}$ are positive constant, and therefore $\Vert\nabla v\Vert_{L^2}$ is bounded for all $t\in(0,T_{\max})$ as desired.
\end{proof}

\section{Parabolic--parabolic system in multi--dimensional domain}\label{section3}
According to Lemma \ref{proposition21} and (\ref{23}), in order to prove Theorem \ref{theorem11} and Theorem \ref{theorem12}, it is sufficient to show that $\Vert u(\cdot,t)\Vert_{L^\infty(\Omega)}$ is bounded for $t\in(0,T_{\max})$ and therefore $T_{\max}=\infty$ and the solution is global.  Indeed we will show that $\Vert u(\cdot,t)\Vert_{L^\infty(\Omega)}$ is uniformly bounded in $t\in(0,\infty)$.  To this end, it is sufficient to prove that $\Vert u(\cdot,t)\Vert_{L^p(\Omega)}$ is bounded for some $p$ large according to (\ref{25}).  For this purpose we will give a combined estimate on $\int_{\Omega} u^p+\int_{\Omega} |\nabla v|^{2q}$ for both $p$ and $q$ large based on the idea recently developed in \cite{LTW,TW2,Winkler2} etc.  That being said, we will first prove the boundedness of $\int_{\Omega} u^p$ in terms of a functional involving $\int_{\Omega} |\nabla v|^{2q}$ in Lemma \ref{lemma31} and Lemma \ref{lemma32}.  Choosing $p>N$, one obtains from (\ref{25}) the boundedness of $\Vert v \Vert_{W^{1,\infty}}$, and then that of $\Vert u \Vert_{L^{\infty}}$ easily follows from the standard Moser--Alikakos $L^p$--iteration in \cite{A0}.

\subsection{A priori estimates}
For any $p\geq2$, we multiply the $u$-equation in (\ref{11}) by $u^{p-1}$ and then integrate it over $\Omega$ by parts
\begin{align}\label{31}
\frac{1}{p}\frac{d}{dt}\int_{\Omega} u^p=&\int_{\Omega} u^{p-1} \nabla\cdot(D_1(u)\nabla u)+\int_{\Omega} u^{p-1} \nabla\cdot(\chi \phi(u)\nabla v)+\int_{\Omega} u^p(a_1-b_1u^\alpha-c_1v)      \nonumber\\
=&-(p-1)\int_{\Omega} D_1(u)u^{p-2}|\nabla u|^2-(p-1)\int_{\Omega} \chi \phi(u)u^{p-2}\nabla u\nabla v  \nonumber\\
&+\int_{\Omega} u^p(a_1-b_1u^\alpha-c_1v).
\end{align}
In light of $D_1(u)\geq M_1(1+u)^{m_1}>M_1 u^{m_1}$, we have
\begin{equation}\label{32}
(p-1)\int_{\Omega} D_1(u)u^{p-2}|\nabla u|^2
\geq M_1(p-1)\int_{\Omega} u^{p+m_1-2}|\nabla u|^2
= \frac{4M_1(p-1)}{(p+m_1)^2}\int_{\Omega}|\nabla u^{\frac{p+m_1}{2}}|^2,
\end{equation}
where the identity follows from
\[ u^{p+m_1-2}|\nabla u|^2=\frac{4}{(p+m_1)^2} |\nabla u^{\frac{p+m_1}{2}}|^2.\]
Moreover, Young's inequality implies
\begin{align}\label{33}
  & -(p-1)\int_{\Omega} \chi \phi(u)u^{p-2}\nabla u\nabla v \nonumber \\
\leq &\frac{M_1(p-1)}{2}\int_{\Omega} u^{p+m_1-2}|\nabla u|^2+ \frac{\chi^2(p-1)}{2M_1} \int_{\Omega} u^{p-m_1-2}\phi^2(u)|\nabla v|^2 \nonumber \\
\leq &\frac{2M_1(p-1)}{(p+m_1)^2}\int_{\Omega}|\nabla u^{\frac{p+m_1}{2}}|^2+\frac{\chi^2 M_2^2(p-1)}{2M_1} \int_{\Omega} u^{p-m_1+2m_2-2}|\nabla v|^2
\end{align}
and
\begin{align}\label{34}
\Big(a_1+\frac{1}{p}\Big)\int_{\Omega}u^{p}\leq \frac{b_1}{2}\int_{\Omega}u^{p+\alpha}+C_{31},
\end{align}
where $C_{31}$ is a positive constant dependent on $p$.  Thanks to (\ref{32}), (\ref{33}) and (\ref{34}), we have from (\ref{31})
\begin{align}\label{35}
  &\frac{1}{p}\frac{d}{dt}\int_{\Omega} u^p +\frac{1}{p}\int_{\Omega} u^p +\frac{2M_1(p-1)}{(p+m_1)^2}\int_{\Omega} |\nabla u^{\frac{p+m_1}{2}}|^2+\frac{b_1}{2}\int_{\Omega}u^{p+\alpha}    \nonumber\\
\leq& \frac{\chi^2 M_2^2(p-1)}{2M_1} \int_{\Omega} u^{p-m_1+2m_2-2}|\nabla v|^2 +C_{31}.
\end{align}

On the other hand, for any $q>1$, we have from the $v$-equation in (\ref{11})
\begin{align}\label{36}
&\frac{1}{2q}\frac{d}{dt}\int_{\Omega} |\nabla v|^{2q}
=\int_{\Omega} |\nabla v|^{2q-2} \nabla v \cdot\nabla v_t  \nonumber \\
=&\smash[b]{\overbrace{ D_2\int_{\Omega} |\nabla v|^{2q-2} \nabla v\cdot \nabla\Delta v\,}^\text{$I_1$}}+
\smash[b]{\overbrace{\int_{\Omega} |\nabla v|^{2q-2} \nabla v\cdot \nabla [(a_2-b_2u-c_2v)v]\,}^\text{$I_2$}}.
\end{align}
In light of the identity
\[\nabla v\cdot \nabla \Delta v=\frac{1}{2}\Delta |\nabla v|^2- |D^2 v|^2,\]
we first estimate $I_{1}$ in (\ref{36}) through
\begin{align}\label{37}
I_1=&\frac{D_2}{2}\int_{\Omega} |\nabla v|^{2q-2} \Delta|\nabla v|^2-D_2\int_{\Omega} |\nabla v|^{2q-2} |D^2 v|^2 \nonumber\\
=&\frac{D_2}{2}\int_{\partial\Omega} |\nabla v|^{2q-2} \frac{\partial|\nabla v|^2}{\partial n}-\frac{D_2}{2}\int_{\Omega}\nabla |\nabla v|^{2q-2}\cdot \nabla|\nabla v|^2   -D_2\int_{\Omega} |\nabla v|^{2q-2} |D^2 v|^2       \nonumber\\
=&\smash[b]{\overbrace{\frac{D_2}{2}\int_{\partial\Omega} |\nabla v|^{2q-2} \frac{\partial|\nabla v|^2}{\partial n}\,}^\text{$I_{11}$}}-
\smash[b]{\overbrace{ \frac{(q-1)D_2}{2}\int_{\Omega} |\nabla v|^{2q-4}  \Big| \nabla|\nabla v|^2 \Big|^2\,}^\text{$I_{12}$}}\nonumber \\
&-\smash[b]{\overbrace{D_2\int_{\Omega} |\nabla v|^{2q-2} |D^2 v|^2 \,}^\text{$I_{13}$}}.
\end{align}

To further estimate $I_{11}$, we invoke the inequality $\frac{\partial|\nabla v|^2}{\partial n}\leq C_{\Omega}|\nabla v|^2$ (e.g. inequality (2.4) in \cite{ISY} due to \cite{HT})
with $C_{\Omega}$ being a positive constant depending only on the curvatures of $\partial\Omega$ to deduce
\begin{equation}\label{38}
I_{11}=\frac{D_2}{2}\int_{\partial\Omega} |\nabla v|^{2q-2} \frac{\partial|\nabla v|^2}{\partial n}\leq  \frac{D_2 C_{\Omega}}{2}\int_{\partial\Omega} |\nabla v|^{2q}
  :=C_{\Omega}\||\nabla v|^q\|^2_{L^2(\partial\Omega)}.
\end{equation}
By taking $r\in(0,\frac{1}{2})$, we have from (1.9) in \cite{ISY} that the trace embedding $ W^{r+\frac{1}{2},2}(\Omega) (\hookrightarrow W^{r,2}(\partial\Omega))\hookrightarrow L^2(\partial\Omega)$ is compact and therefore there exists a positive constant $C_{32}$ such that
\begin{align}\label{39}
  \||\nabla v|^q\|_{L^2(\partial\Omega)}\leq C_{32}\||\nabla v|^q\|_{W^{r+\frac{1}{2},2}(\Omega)}.
\end{align}
Let $h_1\in(0,1)$ satisfy
\[\frac{1}{2}-\frac{r+\frac{1}{2}}{N}=\Big(1-h_1\Big)\frac{q}{s}+h_1\Big(\frac{1}{2}-\frac{1}{N}\Big),\]
or
\[h_1=\frac{\frac{q}{s}-(\frac{1}{2}-\frac{1}{2N}-\frac{r}{N})}{\frac{q}{s}-(\frac{1}{2}-\frac{1}{N})}\in\Big(r+\frac{1}{2},1\Big),\]
where we choose some $s\in [1,\frac{N}{N-1})$ if $\alpha<1$ and $s=$ if $\alpha\geq1$, then we invoke the fractional Gagliardo--Nirenberg interpolation inequality to deduce
\begin{align}\label{310}
  \||\nabla v|^q\|_{W^{r+\frac{1}{2},2}(\Omega)}&\leq C_{33} \|\nabla|\nabla v|^q\|^{h_1}_{L^2(\Omega)}\||\nabla v|^q\|^{1-h_1}_{L^{\frac{s}{q}}(\Omega)}+C_{34}\||\nabla v|^q\|_{L^{\frac{s}{q}}(\Omega)}\nonumber\\
  & \leq C_{35} \|\nabla|\nabla v|^q\|^{h_1}_{L^2(\Omega)}+C_{36},
\end{align}
where we have applied the fact that $\Vert \nabla v\Vert_{L^s}$ is uniformly bounded.  In conjunction with (\ref{39}) and (\ref{310}), we apply Young's inequality and the fact that $h_1<1$ in (\ref{38}) to obtain
\begin{align}\label{311}
I_{11}
&\leq 2C_\Omega C^2_{32} (C^2_{35} \|\nabla|\nabla v|^q\|^{2h_1}_{L^2(\Omega)}+ C^2_{36})\nonumber \\
&\leq \frac{(q-1)D_2}{q^2} \int_\Omega\Big|\nabla|\nabla v|^q\Big|^2+ C_{37}
\end{align}
where $C_{37}$ is a positive constant.  To further estimate $I_{12}$, we note that
\[|\nabla v|^{2q-4}\Big|\nabla|\nabla v|^2\Big|^2=\frac{4}{q^2}\Big|\nabla|\nabla v|^q\Big|^2,\]
then
\begin{equation}\label{312}
I_{12}=\frac{2(q-1)D_2}{q^2}\int_\Omega \Big|\nabla|\nabla v|^q\Big|^2.
\end{equation}
Substituting (\ref{311}) and (\ref{312}) into (\ref{37}) gives us
\begin{equation}\label{313}
I_1\leq  -\frac{(q-1)D_2}{q^2}\int_\Omega\Big|\nabla|\nabla v|^q\Big|^2   -D_2\int_{\Omega} |\nabla v|^{2q-2} |D^2 v|^2 +C_{38}.
\end{equation}

To estimate $I_2$, we obtain from the integration by parts
\begin{align}\label{314}
I_2=&\int_{\Omega} |\nabla v|^{2q-2} \nabla v \cdot \nabla[(a_2-b_2u-c_2v)v]      \nonumber \\
=&-\int_{\Omega} (a_2-b_2u-c_2v)v \nabla \cdot(|\nabla v|^{2q-2} \nabla v)       \nonumber  \\
=&-\int_{\Omega} (a_2-b_2u-c_2v)v |\nabla v|^{2q-2} \Delta v \nonumber  \\
&- (q-1)\int_{\Omega} (a_2-b_2u-c_2v)v |\nabla v|^{2q-4} \nabla|\nabla v|^2\cdot \nabla v   \nonumber\\
=&-\smash[b]{\overbrace{ \int_{\Omega} (a_2-c_2v)v |\nabla v|^{2q-2} \Delta v\,}^\text{$I_{21}$}}-\smash[b]{\overbrace{ (q-1)\int_{\Omega} (a_2-c_2v)v |\nabla v|^{2q-4} \nabla|\nabla v|^2\cdot \nabla v\,}^\text{$I_{22}$}}              \nonumber\\
&+\smash[b]{\overbrace{b_2\int_{\Omega} uv|\nabla v|^{2q-2} \Delta v\,}^\text{$I_{23}$}}
+\smash[b]{\overbrace{(q-1)b_2\int_{\Omega} uv |\nabla v|^{2q-4} \nabla|\nabla v|^2\cdot \nabla v\,}^\text{$I_{24}$}}. \\ \nonumber
\end{align}
Moreover we apply Young's inequality to have
\begin{align}\label{315}
-I_{21} &\leq \frac{D_2}{2N} \int_{\Omega} |\nabla v|^{2q-2} |\Delta v|^2 +\frac{N}{2D_2}\int_{\Omega} (a_2-c_2v)^2v^2|\nabla v|^{2q-2}         \nonumber\\
&\leq \frac{D_2}{2} \int_{\Omega} |\nabla v|^{2q-2} |D^2 v|^2 +C_{39}\int_{\Omega}|\nabla v|^{2q-2},
\end{align}
where $C_{39}$ is a positive constant and the second inequality comes from $|\Delta v|^2\leq N |D^2 v|^2$.  Similarly we have
\begin{align}\label{316}
-I_{22} \leq& \frac{(q-1)D_2}{16}
\int_{\Omega} |\nabla v|^{2q-4}  \Big| \nabla|\nabla v|^2 \Big|^2  +C_{310}\int_{\Omega}|\nabla v|^{2q-2},
\end{align}
\begin{align}\label{317}
I_{23} &\leq \frac{D_2}{2} \int_{\Omega} |\nabla v|^{2q-2} |D^2 v|^2 +C_{311}\int_{\Omega} u^2|\nabla v|^{2q-2},
\end{align}
and
\begin{align}\label{318}
I_{24} \leq \frac{(q-1)D_2}{16}\int_{\Omega} |\nabla v|^{2q-4}  \Big| \nabla|\nabla v|^2 \Big|^2  +C_{312}\int_{\Omega} u^2|\nabla v|^{2q-2},
\end{align}
with positive constants $C_{310}, C_{311}$ and $ C_{312}$.   Collecting (\ref{315})--(\ref{318}), we infer from (\ref{314})
\begin{align}\label{319}
I_2
\leq& D_2\int_{\Omega} |\nabla v|^{2q-2} |D^2 v|^2+\frac{(q-1)D_2}{8}\int_{\Omega} |\nabla v|^{2q-4}  \Big| \nabla|\nabla v|^2 \Big|^2\nonumber\\
&+(C_{39}+C_{310})\int_{\Omega} |\nabla v|^{2q-2}+(C_{311}+C_{312})\int_{\Omega} u^2|\nabla v|^{2q-2}\nonumber\\
=& D_2\int_{\Omega} |\nabla v|^{2q-2} |D^2 v|^2+\frac{(q-1)D_2}{2q^2}\int_\Omega \Big|\nabla|\nabla v|^q\Big|^2\nonumber\\
&+(C_{39}+C_{310})\int_{\Omega} |\nabla v|^{2q-2}+(C_{311}+C_{312})\int_{\Omega} u^2|\nabla v|^{2q-2}.
\end{align}
Combining (\ref{319}) with (\ref{313}), we have from (\ref{36})
\begin{align*}
\frac{1}{2q}\frac{d}{dt}\int_{\Omega} |\nabla v|^{2q}\leq &
-\frac{(q-1)D_2}{2q^2}\int_\Omega\Big|\nabla|\nabla v|^q\Big|^2+(C_{39}+C_{310})\int_{\Omega} |\nabla v|^{2q-2}  \nonumber\\
&+(C_{311}+C_{312})\int_{\Omega} u^2|\nabla v|^{2q-2}+C_{38}
\end{align*}
or equivalently
\begin{align}\label{320}
&\frac{1}{2q} \frac{d}{dt}\int_{\Omega} |\nabla v|^{2q} +\frac{1}{2q} \int_{\Omega} |\nabla v|^{2q}
+\frac{(q-1)D_2}{2q^2}\int_{\Omega}\Big| \nabla|\nabla v|^q \Big|^2    \nonumber\\
\leq&C_{313} \int_{\Omega} |\nabla v|^{2q}+C_{314}\int_{\Omega} u^2|\nabla v|^{2q-2}+C_{38},
\end{align}
where $C_{313}=C_{39}+C_{310}+\frac{1}{2q}$ and $C_{314}=C_{311}+C_{312}$.  Using Gagliardo--Nirenberg interpolation inequality we estimate
\begin{align}\label{321}
  C_{313}\int_{\Omega} |\nabla v|^{2q}=&C_{313}\Big\||\nabla v|^q\Big\|^2_{L^2(\Omega)}\nonumber\\
  \leq & C_{315}\Big\|\nabla|\nabla v|^q\Big\|^{2h_2}_{L^2(\Omega)}\Big\||\nabla v|^q\Big\|^{2(1-h_2)}_{L^{\frac{s}{q}}(\Omega)}+C_{315}\Big\||\nabla v|^q\Big\|^{2}_{L^{\frac{s}{q}}(\Omega)} \nonumber \\
  \leq & \frac{(q-1)D_2}{4q^2}\Big\|\nabla|\nabla v|^q\Big\|^{2}_{L^2(\Omega)}+C_{316},
\end{align}
thanks to
\[h_2=\frac{\frac{q}{s}-\frac{1}{2}}{\frac{q}{s}-(\frac{1}{2}-\frac{1}{N})}\in (0,1),\]
and the boundedness of $\||\nabla v|^q\|^{2}_{L^{\frac{s}{q}}(\Omega)} =\||\nabla v|\|^{2q}_{L^s(\Omega)}$ in (\ref{24}).

Again, thanks to Young's inequality, (\ref{320}) and (\ref{321}) imply
\begin{align}\label{322}
&\frac{1}{2q}\frac{d}{dt}\int_{\Omega} |\nabla v|^{2q} +\frac{1}{2q}\int_{\Omega} |\nabla v|^{2q}
+\frac{(q-1)D_2}{4q^2}\int_{\Omega}\Big| \nabla|\nabla v|^q \Big|^2   \nonumber \\
\leq& C_{314} \int_{\Omega} u^2|\nabla v|^{2q-2}+C_{317}.
\end{align}
Finally by collecting (\ref{35}) and (\ref{322}) we conclude
\begin{align}\label{323}
  &\frac{d}{dt}\Big( \frac{1}{p}\int_{\Omega} u^p + \frac{1}{2q}\int_{\Omega} |\nabla v|^{2q}        \Big)
  +\Big( \frac{1}{p}\int_{\Omega} u^p + \frac{1}{2q}\int_{\Omega} |\nabla v|^{2q}        \Big)   \nonumber\\
+  &\frac{2M_1(p-1)}{(p+m_1)^2}\int_{\Omega} |\nabla u^{\frac{p+m_1}{2}}|^2
  +\frac{(q-1)D_2}{4q^2}\int_{\Omega}\Big| \nabla|\nabla v|^q \Big|^2+\frac{b_1}{2}\int_{\Omega}u^{p+\alpha} \nonumber\\
\leq &\frac{\chi^2M_2^2(p-1)}{2M_1}  \overbrace{\int_{\Omega} u^{p-m_1+2m_2-2}|\nabla v|^2}^{I_{31}}+
   C_{314}\overbrace{\int_{\Omega} u^2|\nabla v|^{2q-2}}^{I_{32}}+C_{317}.
\end{align}

\begin{lemma}\label{lemma31}
Let $(u,v)$ be a positive classical solution of (\ref{11}) in $\Omega\times(0, T_{\max})$.  Suppose that $m_1$ and $m_2$ satisfy condition (\ref{14}).  Then for large $p$ and $q$ there exists a positive constant $C(p,q)$ such that
\begin{equation}\label{324}
\int_{\Omega} u^p+\int_{\Omega} |\nabla v|^{2q}\leq C(p,q),\forall t\in(0,T_{\max}).
\end{equation}
\end{lemma}
\begin{proof}
For the consistency of notation we denote
\begin{equation}\label{325}
\lambda_1=p-m_1+2m_2-2, \lambda_2=2,
\end{equation}
and
\begin{equation}\label{326}
\kappa_1=2, \kappa_2=2(q-1).
\end{equation}
Let $\mu_{i}>1$ be an arbitrary real number and $\mu_i':=\frac{\mu_i}{\mu_i-1}$ be its conjugate, then we can apply H\"older's inequality to estimate $I_{3i}$ in (\ref{323}) by
\[I_{31} \leq \Big( \int_{\Omega}u^{(p-m_1+2m_2-2)\mu_1}\Big)^{\frac{1}{\mu_1}}\cdot \Big( \int_{\Omega}|\nabla v|^{2 \mu_1'}\Big)^{\frac{1}{\mu_1'}}
:= \Big( \int_{\Omega}u^{\lambda_1 \mu_1}\Big)^{\frac{1}{\mu_1}}\cdot \Big( \int_{\Omega}|\nabla v|^{\kappa_1 \mu_1'}\Big)^{\frac{1}{\mu_1'}}\]
and
\[I_{32}
\leq\Big( \int_{\Omega}u^{2\mu_2}\Big)^{\frac{1}{\mu_2}}\cdot \Big( \int_{\Omega}|\nabla v|^{2(q-1)\mu_2'}\Big)^{\frac{1}{\mu_2'}}
:= \Big( \int_{\Omega}u^{\lambda_2 \mu_2}\Big)^{\frac{1}{\mu_2}}\cdot \Big( \int_{\Omega}|\nabla v|^{\kappa_2 \mu_2'}\Big)^{\frac{1}{\mu_2'}},\]
which can be simplified as
\begin{equation}\label{327}
I_{3i}\leq \Big( \int_{\Omega}u^{\lambda_i \mu_i}\Big)^{\frac{1}{\mu_i}}\cdot \Big( \int_{\Omega}|\nabla v|^{\kappa_i \mu_i'}\Big)^{\frac{1}{\mu_i'}}, i=1,2.
\end{equation}

By Gagliardo--Nirenberg interpolation inequality, there exist positive constants $C_{318}$ and $C_{319}$ such that
\begin{align}\label{328}
&\Big(\int_{\Omega}u^{\lambda_i \mu_i}\Big)^{\frac{1}{\mu_i}}=\Big\| u^{\frac{p+m_1}{2}} \Big\|_{L^{\frac{2\lambda_i \mu_i}{p+m_1}}(\Omega)}^{\frac{2\lambda_i}{p+m_1}}     \nonumber\\
\leq& C_{318} \Big\| \nabla u^{\frac{p+m_1}{2}} \Big\|_{L^2(\Omega)}^{\frac{2\lambda_i}{p+m_1}\cdot h_{3i}}\cdot\Big\| u^{\frac{p+m_1}{2}} \Big\|_{L^{\frac{2}{p+m_1}}(\Omega)}^{\frac{2\lambda_i}{p+m_1}\cdot(1-h_{3i})} +C_{318}\Big\| u^{\frac{p+m_1}{2}} \Big\|_{L^{\frac{2}{p+m_1}}(\Omega)}^{\frac{2\lambda_i}{p+m_1}}     \nonumber\\
\leq& C_{319}\Big\| \nabla u^{\frac{p+m_1}{2}} \Big\|_{L^2(\Omega)}^{\frac{2\lambda_i}{p+m_1}\cdot h_{3i}}+C_{319}
\end{align}
with
\begin{align}\label{329}
h_{3i}:=\frac{\frac{p+m_1}{2}-\frac{p+m_1}{2\lambda_i \mu_i}}{\frac{p+m_1}{2}-(\frac{1}{2}-\frac{1}{N})}
\end{align}
and
\begin{align}\label{0330}
& \Big( \int_{\Omega}|\nabla v|^{2\mu_i'}\Big)^{\frac{1}{\mu_i'}} =\Big\| |\nabla v|^{q} \Big\|_{L^{\frac{\kappa_i \mu_i'}{q}}(\Omega)}^{\frac{\kappa_i}{q}}    \nonumber\\
\leq& C_{320}  \Big\| \nabla|\nabla v|^{q} \Big\|_{L^2(\Omega)}^{\frac{\kappa_i}{q}\cdot h_{4i}}\cdot\Big\| |\nabla v|^{q} \Big\|_{L^{\frac{s}{q}}(\Omega)}^{\frac{\kappa_i}{q}\cdot(1-h_{4i})} +C_{320}\Big\| |\nabla v|^{q} \Big\|_{L^{\frac{s}{q}}(\Omega)}^{\frac{\kappa_i}{q}}   \nonumber\\
\leq& C_{321}\Big\| \nabla|\nabla v|^{q} \Big\|_{L^2(\Omega)}^{\frac{\kappa_i}{q}\cdot h_{4i}}+C_{321}
\end{align}
with
\begin{align}\label{0331}
h_{4i}:=\frac{\frac{q}{s}-\frac{q}{\kappa_i \mu_i'}}{\frac{q}{s}-(\frac{1}{2}-\frac{1}{N})},
\end{align}
where we have used in (\ref{328}) and (\ref{0330}) the boundedness of $\Vert u\Vert_{L^1}$ and $\Vert \nabla v\Vert_{L^s}$ from (\ref{22}) and (\ref{24}) respectively.

We now claim that there always exist $\mu_i>1$, $i=1,2$ such that
\begin{equation}\label{0332}
\frac{2\lambda_i \mu_i}{p+m_1}\geq1, \frac{\kappa_i\mu'_i}{q}\geq1, 0<h_{3i}, h_{4i}<1
\end{equation}
and under condition (\ref{14})
\begin{align}\label{0333}
f_i(p,q,s):=\frac{2\lambda_i}{p+m_1}\cdot h_{3i}+\frac{\kappa_i}{q}\cdot h_{4i} =\frac{\lambda_i-\frac{1}{\mu_i}}{\frac{p+m_1}{2}-(\frac{1}{2}-\frac{1}{N})}
+\frac{\frac{\kappa_i}{s}-\frac{1}{\mu_i'}}{\frac{q}{s}-(\frac{1}{2}-\frac{1}{N})}<2.
\end{align}

On the other hand, we recall that if $\alpha+\beta<2$, then for any $\epsilon>0$ there exists $C_\epsilon>0$ such that $(x^\alpha+1)(y^\beta+1)\leq \epsilon(x^2+y^2)+C_\epsilon$ for any $x,y\in \mathbb R^+$.  Therefore, if conditions (\ref{0332}) and (\ref{0333}) hold, we can have
\begin{align}\label{0334}
I_{3i}\leq & \Big( \int_{\Omega} |\nabla u^{\frac{p+m_1}{2}}|^2 \Big)^{\frac{1}{2}\cdot \frac{2\lambda_i}{p+m_1} \cdot h_{3i}}
\cdot \Big( \int_{\Omega} \Big\| \nabla|\nabla v|^q \Big\|^2\Big)^{\frac{1}{2}\cdot \frac{\kappa_i}{q}\cdot (1-h_{3i})}+C_{322}
\nonumber\\
\leq &  \epsilon\int_{\Omega} |\nabla u^{\frac{p+m_1}{2}}|^2 +\epsilon\int_{\Omega}\Big| \nabla|\nabla v|^q \Big|^2+C_{322}.
\end{align}
Combining (\ref{323}) with (\ref{0334}), we conclude that
\begin{align}\label{0335}
\frac{d}{dt} \Big(\frac{1}{p}\int_{\Omega} u^p +\frac{1}{2q}\int_{\Omega} |\nabla v|^{2q} \Big)+\Big(\frac{1}{p}\int_{\Omega} u^p +\frac{1}{2q}\int_{\Omega} |\nabla v|^{2q}\Big)\leq C(p,q)
\end{align}
for all $t\in (0,\infty)$.  Then we can apply the Gr\"{o}nwall's lemma on (\ref{0335}) to show (\ref{324}).

Now in order to complete the proof of Lemma \ref{lemma31}, we only need to verify (\ref{0332}) and (\ref{0333}) claimed above in order to apply the Gagliardo--Nirenberg--Sobolev inequality.  First of all, we see that (\ref{0332}) is equivalent as
\[\frac{1}{2}-\frac{1}{N}<\frac{p+m_1}{2\lambda_i\mu_i}\leq 1 \text{~and~} \frac{1}{2}-\frac{1}{N}<\frac{q}{\kappa_i\mu_i'}\leq 1,\]
which, in view of (\ref{325}) and (\ref{326}), become
\begin{equation}\label{0336}
\frac{1}{2}-\frac{1}{N}<\frac{p+m_1}{2(p-m_1+2m_2-2)\mu_1}\leq 1, \quad \frac{1}{2}-\frac{1}{N}<\frac{q}{2\mu_1'}\leq 1
\end{equation}
and
\begin{equation}\label{0337}
\frac{1}{2}-\frac{1}{N}<\frac{p+m_1}{4\mu_2}\leq 1, \quad \frac{1}{2}-\frac{1}{N}<\frac{q}{2(q-1)\mu_2'}\leq 1.
\end{equation}
In the sequel we choose $\mu_1:=\mu_1(q)=\frac{q}{q-1}$ and $\mu_2:=\mu_2(p)=\frac{p}{2}$, and then it is easy to see that (\ref{0336}) and (\ref{0337}) hold for large $p$ and $q$.

Finally we are left to prove that $f_i(p,q,s)<2$ in (\ref{0333}) which, in light of (\ref{325}) and (\ref{326}), are
\[f_1(p,q,s)=\frac{p-m_1+2m_2-2-\frac{1}{\mu_1}}{\frac{p+m_1}{2}-(\frac{1}{2}-\frac{1}{N})}
+\frac{\frac{2}{s}-\frac{1}{\mu_1'}}{\frac{q}{s}-(\frac{1}{2}-\frac{1}{N})}<2\]
and
\[f_2(p, q,s)=\frac{2-\frac{1}{\mu_2}}{\frac{p+m_1}{2}-(\frac{1}{2}-\frac{1}{N})}
+\frac{\frac{2(q-1)}{s}-\frac{1}{\mu_2'}}{\frac{q}{s}-(\frac{1}{2}-\frac{1}{N})}<2.\]
By straightforward calculations we see that $f_1(p,q,s)<2$ and $f_2(p,q,s)<2$ are equivalent as
\begin{equation}\label{0338}
\frac{q}{s}>\zeta_1\Big(\frac{p+m_1}{2}-\Big(\frac{1}{2}-\frac{1}{N}\Big)\Big)+\Big(\frac{1}{2}-\frac{1}{N}\Big)
\end{equation}
and
\begin{equation}\label{0339}
\frac{q}{s}<\zeta_2\Big(\frac{p+m_1}{2}-\Big(\frac{1}{2}-\frac{1}{N}\Big)\Big)+\Big(\frac{1}{2}-\frac{1}{N}\Big),
\end{equation}
where \[\zeta_1=\zeta_1(p,q,s):=\frac{\frac{1}{s}-\frac{1}{2\mu'_1(p,q)}}{m_1-m_2+\frac{1}{2}+\frac{1}{N}+\frac{1}{2\mu_1(p,q)}}>0\]
and \[\zeta_2=\zeta_2(p,q,s):=\frac{\frac{1}{s}+\frac{1}{2\mu'_2(p,q)}-(\frac{1}{2}-\frac{1}{N}) }{1-\frac{1}{2\mu_2(p,q)}}>0.\]
We want to mention that the denominator in $\zeta_1$ is positive under condition (\ref{14}).

Note that we choose $\mu_1=\frac{q}{q-1}$ and $\mu_2=\frac{p}{2}$ and our discussions are divided into the followings:  case \emph{(i)}. $0<\alpha<1$ and therefore $s\in[1,\frac{N}{N-1})$.  Then
\begin{align*}
   &\zeta_2(\infty,\infty,N/(N-1))-\zeta_1(\infty,\infty,N/(N-1))\\
=&\Big(1-\frac{1}{N}+\frac{1}{2}-(\frac{1}{2}-\frac{1}{N})\Big)-\frac{1-\frac{1}{N}}{m_1-m_2+1+\frac{1}{N}} \\
  =& \frac{m_1-m_2+\frac{2}{N}}{m_1-m_2+1+\frac{1}{N}}>0.
\end{align*}
This implies that, by the continuity of $\zeta_i$, we can always choose $p^*$ and $q^*$ large and $s^*$ smaller than but sufficiently close to $\frac{N}{N-1}$ such that $\zeta_2(p^*,q^*,s^*)>\zeta_1(p^*,q^*,s^*)$, therefore both (\ref{0338}) and (\ref{0339}) hold for such $(p^*,q^*,s^*)$ hence $f_i(p^*,q^*,s^*)<2$.

case \emph{(ii)}.  $\alpha\geq1$ and therefore $s=2$.  Then
\[\zeta_2(\infty,\infty,2)-\zeta_1(\infty,\infty,2)=\frac{m_1-m_2+\frac{3N+2}{N(N+2)}}{(m_1-m_2+1+\frac{1}{N})(\frac{1}{2}+\frac{1}{N})}.\]
Similar as in case \emph{(i)} we have that $\zeta_2(p,q,2)>\zeta_1(p,q,2)$ hence $f_i(p,q,s)<2$ when $p$, $q$ are large.  In both cases (\ref{0333}) holds for large $p$, $q$ under condition (\ref{14}) and this completes the proof of Lemma \ref{lemma31}.
\end{proof}

In the following lemma, we estimate $I_{31}$ and $I_{32}$ by using Young's inequality instead of H\"older's as in Lemma \ref{lemma31}.  We shall see that $\alpha$ plays an important role in \emph{a priori} estimates.
\begin{lemma}\label{lemma32}
Suppose that
\begin{align}\label{338}
 2m_2-m_1<
    \begin{cases}
      \max\{\alpha, m_1\}+\frac{2}{N}, &\text{~if~} 0<\alpha < 1,\\
      \max\{\alpha, m_1\}+\frac{4}{N+2}, & \text{~if~} \alpha\geq1,\\
    \end{cases}
\end{align}
then for large $p$ and $q$ there exists a constant $C(p,q)>0$ such that
\begin{equation}\label{339}
\int_{\Omega} u^p+\int_{\Omega} |\nabla v|^{2q}\leq C(p,q),\forall t\in(0,\infty).
\end{equation}
\end{lemma}
\begin{proof}
First of all, we invoke the Gagliardo--Nirenberg interpolation inequality
\begin{align}\label{340}
\int_{\Omega} u^{p+m_1}&=\Vert u^{\frac{p+m_1}{2}} \Vert_{L^{2}(\Omega)}^{2}     \nonumber\\
&\leq C_{323} \Vert \nabla u^{\frac{p+m_1}{2}} \Vert_{L^2(\Omega)}^{2 h_5}\cdot \Vert u^{\frac{p+m_1}{2}} \Vert_{L^{\frac{2}{p+m_1}}(\Omega)}^{2 (1-h_5)}+C_{323} \Vert u^{\frac{p+m_1}{2}} \Vert_{L^{\frac{2}{p+m_1}}(\Omega)}^{2}                \nonumber\\
&\leq C_{324} \Vert \nabla u^{\frac{p+m_1}{2}} \Vert_{L^2(\Omega)}^{2 h_5}+C_{324},
\end{align}
where we have applied the fact that $\Vert u\Vert_{L^1}$ is bounded and
\[h_5: =\frac{\frac{p+m_1}{2}-\frac{1}{2}}{\frac{p+m_1}{2}-(\frac{1}{2}-\frac{1}{N})}\in(0,1).\]

By Young's inequality, there exists a positive constant $C_{323}$ such that in (\ref{323})
\begin{align}\label{341}
\frac{\chi^2M_2^2(p-1)}{2M_1}I_{31}&\leq \frac{b_1}{4} \int_{\Omega} (u^{p-m_1+2m_2-2})^{\frac{p+\max\{\alpha, m_1\} }{p-m_1+2m_2-2 }} + C_{325} \int_{\Omega}\vert \nabla v \vert ^{2\cdot \frac{p+\max\{\alpha, m_1\}}{\max\{\alpha, m_1\}+m_1-2m_2+2}} \nonumber\\
  &=\frac{b_1}{4}\int_{\Omega}u^{p+\max\{\alpha, m_1\}}+C_{325} \int_{\Omega}\vert \nabla v \vert ^{\theta_1}
\end{align}
and
\begin{align}\label{342}
C_{314}I_{32}&\leq \frac{b_1}{4} \int_{\Omega} (u^{2})^{\frac{p+\max\{\alpha, m_1\}}{2}} + C_{326} \int_{\Omega}\vert \nabla v \vert ^{2(q-1)\cdot \frac{p+\max\{\alpha, m_1\}}{p+\max\{\alpha, m_1\}-2}} \nonumber \\
  & =\frac{b_1}{4} \int_{\Omega} u^{p+\max\{\alpha, m_1\}} + C_{326} \int_{\Omega}\vert \nabla v \vert ^{\theta_2},
\end{align}
where we denote
\begin{equation}\label{343}
\theta_1: = \theta_1(p,q)=\frac{2(p+\max\{\alpha, m_1\})}{\max\{\alpha, m_1\}+m_1-2m_2+2}
\end{equation}
and
\begin{equation}\label{344}
\theta_2: = \theta_2(p,q)=\frac{2(q-1)(p+\max\{\alpha, m_1\})}{p+\max\{\alpha, m_1\}-2}.
\end{equation}
We want to mention that $\theta_i$ are well--defined since $\max\{\alpha, m_1\}>2m_2-m_1-2$ thanks to (\ref{338}).  Substituting (\ref{341})--(\ref{344}) into (\ref{323}), we derive
\begin{align}\label{345}
  &\frac{d}{dt}\Big( \frac{1}{p}\int_{\Omega} u^p + \frac{1}{2q}\int_{\Omega} |\nabla v|^{2q}        \Big)
   +\Big( \frac{1}{p}\int_{\Omega} u^p + \frac{1}{2q}\int_{\Omega} |\nabla v|^{2q}        \Big)
  +\frac{(q-1)D_2}{4q^2}\int_{\Omega}\Big| \nabla|\nabla v|^q \Big|^2 \nonumber\\
   \leq & C_{325} \int_{\Omega} |\nabla v|^{\theta_1}+C_{326} \int_{\Omega} |\nabla v|^{\theta_2}+C_{327}.
\end{align}
According to Gagliardo--Nirenberg interpolation inequality, we have for $i=1,2$
\begin{align*}
  \int_{\Omega} |\nabla v|^{\theta_i}= &\Big \| \vert \nabla v \vert^q \Big\|_{L^{\frac{\theta_i}{q}}(\Omega)}^{\frac{\theta_i}{q}} \\
  \leq&  C_{328} \Big \| \nabla \vert \nabla v \vert^q \Big\|_{L^{2}(\Omega)}^{\frac{\theta_i}{q}h_{6i}}\Big \| \vert \nabla v \vert^q \Big\|_{L^{\frac{s}{q}}(\Omega)}^{\frac{\theta_i}{q}(1-h_{6i})}+C_{328}\Big \| \vert \nabla v \vert^q \Big\|_{L^{\frac{s}{q}}(\Omega)}^{\frac{\theta_i}{q}}\\
  \leq & C_{329} \Big \| \nabla \vert \nabla v \vert^q \Big\|_{L^{2}(\Omega)}^{\frac{\theta_i}{q}h_{6i}}+C_{330},
\end{align*}
where we have applied the boundedness of $\Vert \nabla v \Vert_{L^s(\Omega)}$ and
\[h_{6i} := h_{6i}(p,q;s)=\frac{\frac{q}{s}-\frac{q}{\theta_i}}{\frac{q}{s}-(\frac{1}{2}-\frac{1}{N})}.\]

Denote
\[g_i(p,q;s):=\frac{\theta_i}{q}h_{6i}(p,q;s).\]
We claim that under condition (\ref{338}) there exists $p$ and $q$ large such that the followings hold
\begin{equation}\label{346}
 0<h_{6i}(p,q;s)<1\text{~and~}0<g_i(p,q;s)<2.
\end{equation}
Assuming (\ref{346}), we conclude from Gagliardo--Nirenberg interpolation inequality and the Young's inequality that for any $\epsilon>0$
\begin{equation}\label{347}
\int_{\Omega} |\nabla v|^{\theta_i}\leq \epsilon \Big \| \nabla \vert \nabla v \vert^q \Big\|_{L^{2}(\Omega)}^{2}+C_\epsilon.
\end{equation}
Substituting (\ref{347}) into (\ref{345}), we can easily derive that
\[y'(t)+y(t)\leq C_{325},\]
by setting $y(t):=\frac{1}{p}\int_{\Omega} u^p + \frac{1}{2q}\int_{\Omega} |\nabla v|^{2q}$ and solving this inequality by Gr\"{o}nwall's lemma gives rise to (\ref{339}).

Now we need to verify the inequalities in (\ref{346}), which by straightforward calculations, are equivalent as
\[\theta_i>s, q>\frac{\theta_i}{2}-\frac{s}{N}.\]
It is easy to see that $\theta_i>s$ hold since both $p$ and $q$ chosen to be large, and therefore we shall only need to verify that $q>\frac{\theta_i}{2}-\frac{s}{N}$ in the sequel.  We divide our discussions into the following two cases: case \emph{(i)}.  $0<\alpha<1$ and therefore $s\in[1,\frac{N}{N-1})$.  Then we can solve the inequalities $q>\frac{\theta_i}{2}-\frac{s}{N}$ for $i=1,2$ to see that
\begin{equation}\label{348}
q_1(p)<q<q_2(p),
\end{equation}
with
\[q_1(p)=\frac{p+\max\{\alpha, m_1\}}{\max\{\alpha, m_1\}+m_1-2m_2+2}-\frac{s}{N}\]
and
\[q_2(p)=\frac{(N+s)(p+\max\{\alpha, m_1\})}{2N}-\frac{s}{N}.\]
Since we shall choose both $p$ and $q$ to be large, we see that (\ref{348}) holds for some $s\in [1,\frac{N}{N-1})$ as long as $\frac{1}{\max\{\alpha, m_1\}+m_1-2m_2+2}<\frac{N+s}{2N}$, or equivalently the condition $2m_2-m_1<\max\{\alpha, m_1\}+\frac{2}{N}$ in (\ref{338}) holds.

case \emph{(ii)}.  $\alpha\geq1$ and therefore $s=2$.  The arguments are the same as in case (i) except that now the condition $q_1(p)<q<q_2(p)$, which implies that $\frac{1}{\max\{\alpha, m_1\}+m_1-2m_2+2}<\frac{N+2}{2N}$, holds provided that $2m_2-m_1<\max\{\alpha, m_1\}+\frac{4}{N+2}$.

Therefore we have verified (\ref{346}) for $p$ and $q$ large under (\ref{338}) and the proof of Lemma \ref{lemma32} completes.
\end{proof}

\subsection{Global existence and boundedness}

\begin{proof}[Proof\nopunct] \emph{of Theorem} \ref{theorem11}.
Taking some $p>N$ fixed, we have from Lemma \ref{lemma22} and \ref{lemma31} that $\Vert v(\cdot,t)\Vert_{W^{1,\infty}}$ is uniformly bounded.  Then one can apply the standard Moser--Alikakos $L^p$ iteration \cite{A0} or the user--friendly version in Lemma A.1 of \cite{TW2} to (\ref{11}) to establish the uniform boundedness of $\Vert u(\cdot,t)\Vert_{L^{\infty}}$.  Therefore the local solution $(u,v)$ is global thanks to the extension criterion in Proposition \ref{proposition21}.  Finally, we can apply the standard parabolic regularity theory to show that $(u,v)$ has the regularity in the theorem.
\end{proof}
\begin{proof}[Proof\nopunct] \emph{of Theorem} \ref{theorem12}.
The proof is the same as that of Theorem \ref{theorem11} in view of Lemma \ref{lemma32}.
\end{proof}

\section{Parabolic--elliptic system in multi--dimensional domain}\label{section4}
In this section, we study the global solutions of parabolic--elliptic system of (\ref{11}).  This describes a competition relationship in which $v$ diffuses much faster than $u$.  We shall prove the global existence and boundedness of the classical solutions of the system.
\subsection{Parabolic--elliptic system with repulsion}
First of all, we consider the parabolic--elliptic system of (\ref{11}) of the following form
\begin{equation}\label{41}
\left\{
\begin{array}{ll}
u_t=\nabla \cdot (D_1(u) \nabla u+\chi \phi(u) \nabla v)+(a_1-b_1u^{\alpha}-c_1v)u,&x \in \Omega,t>0, \\
0=D_2\Delta v+(a_2-b_2u-c_2v)v,&x \in \Omega,t>0, \\
\frac{\partial u}{\partial \textbf{n}}=\frac{\partial v}{\partial \textbf{n}}=0,&x\in\partial \Omega,t>0,\\
u(x,0)=u_0(x)\geq 0, &x\in \Omega.
\end{array}
\right.
\end{equation}
Our first result concerning (\ref{41}) is the following Theorem.
\begin{theorem}\label{theorem41}
Let $\Omega$ be a bounded domain in $\mathbb{R}^N$, $N\geq2$.  Assume that the smooth functions $D_1(u)$ and $\phi(u)$ satisfy (\ref{12}) and (\ref{13}) respectively with
\begin{equation}\label{42}
2m_2-m_1<\max\{\alpha,m_1\}+1.
\end{equation}
Suppose that $u_0\geq,\not \equiv0$ in $\Omega$ and $u_0\in C^2(\Omega)\cap W^{k,p}(\Omega)$ for some $k>1$ and $p>N$.  Then (\ref{41}) admits a unique positive classical solution $(u,v)$ which is uniformly bounded in $\Omega\times (0,\infty)$.
\end{theorem}
\begin{proof}
The proof is very similar as that of Theorem \ref{theorem11}.  First of all, the local existence in $\Omega\times (0,T_{\max})$ follows from the theory Amann in \cite{Am}.  One can easily apply maximum principle and Hopf's lemma to show that $u(x,t)\geq,\not\equiv 0$ in $\Omega\times(0,\infty)$ and $0<v(x)<\frac{a_2}{c_2}$ in $\Omega$.  Moreover, if we can show that $\Vert \nabla u(\cdot,t)\Vert_{L^p}$ is bounded for some $p>N$, then $\Vert \nabla v\Vert_{L^\infty}$ is also bounded after applying the Sobolev embedding $W^{1,p}(\Omega)\hookrightarrow L^\infty(\Omega)$ to the $v$--equation, and therefore one can apply the standard Moser--Alikakos $L^p$--iteration to establish the boundedness of $\Vert \nabla u\Vert_{L^\infty}$.  Finally the regularity of $(u,v)$ follows from parabolic and elliptic embedding theory.

Now we only need to prove the boundedness of $\int_\Omega u^p(\cdot,t)$ for some $p>N$.  Testing the $u$-equation in (\ref{41}) by $u^{p-1}$ and then integrating it over $\Omega$ by parts, we obtain
\begin{align}\label{43}
\frac{1}{p}\frac{d}{dt}\int_{\Omega} u^p=&-(p-1)\int_{\Omega} D_1(u)u^{p-2}|\nabla u|^2-(p-1)\int_{\Omega} \chi \phi(u)u^{p-2}\nabla u\nabla v \nonumber\\
&+\int_{\Omega} u^p(a_1-b_1u^\alpha-c_1v).
\end{align}
Similar as in (\ref{32}) and (\ref{342}) we invoke the Gagliardo--Nirenberg interpolation inequality to obtain
\begin{align}\label{44}
-(p-1)\int_{\Omega} D_1(u)u^{p-2}|\nabla u|^2 \leq -\frac{4M_1(p-1)}{(p+m_1)^2}\int_{\Omega} |\nabla u^{\frac{p+m_1}{2}}|^2+C_{41}
\leq -\xi \int_{\Omega} u^{p+m_1}+C_{42}(\xi),
\end{align}
and apply Young's inequality to have that for any $\gamma>2$
\begin{align}\label{45}
 &-(p-1)\int_{\Omega} \chi \phi(u)u^{p-2}\nabla u\nabla v\nonumber\\
 \leq & \epsilon \int_{\Omega} u^{p+m_1-2}|\nabla u|^2+ \epsilon \int_{\Omega} u^{\frac{(p-m_1+2m_2-2)}{2}\cdot\frac{2\gamma}{\gamma-2}} + C_{\epsilon}\int_{\Omega} |\nabla v|^{\gamma},
\end{align}
where in (\ref{44}) and (\ref{45}) $\xi>0$ and $\epsilon>0$ are arbitrary and $C_{41}$, $C_{42}$ are positive constants.

We invoke the Gagliardo--Nirenberg interpolation inequality and the boundedness of $v$ to obtain
\begin{align}\label{46}
\int_{\Omega}|\nabla v|^{\gamma}=\| \nabla v \|^{\gamma}_{L^{\gamma}(\Omega)}
\leq& C_{43} \| \Delta v \|^{\frac{\gamma}{2}}_{L^{\frac{\gamma}{2}}(\Omega)} \nonumber
\cdot \|  v \|^{\frac{\gamma}{2}}_{L^{\infty}(\Omega)}+C_{44}\| v \|^{\gamma}_{L^{\infty}(\Omega)}\\
\leq& C_{45} \| \Delta v \|^{\frac{\gamma}{2}}_{L^{\frac{\gamma}{2}}(\Omega)}+C_{46},
\end{align}
where $C_{4i}$ are positive constants.

On the other hand, in light of $D_2 \Delta v=-(a_2-b_2 u-c_2 v)v$ and the boundedness of $v$, we have
\begin{align}\label{47}
  \int_{\Omega}|\nabla v|^{\gamma}=\| \nabla v \|^{\gamma}_{L^{\gamma}(\Omega)}
&\leq C_{47} \| u \|^{\frac{\gamma}{2}}_{L^{\frac{\gamma}{2}}(\Omega)}+C_{46}  = C_{47} \int_{\Omega}u^{\frac{\gamma}{2}}+C_{46}.
\end{align}
Thanks to (\ref{44}), (\ref{45}) and (\ref{47}), we derive from (\ref{43})
\begin{align}\label{48}
  \frac{1}{p}\frac{d}{dt}\int_{\Omega} u^p\leq -\xi \int_{\Omega} u^{p+m_1}-b_1 \int_{\Omega} u^{p+\alpha}+\varepsilon \int_{\Omega} u^{\frac{(p-m_1+2m_2-2)\gamma}{\gamma-2}} +C_{\varepsilon}  \int_{\Omega} u^{\frac{\gamma}{2}} +C_{47}.
\end{align}

Choosing $\gamma=2(p-m_1+2m_2-1)$ with $\frac{\gamma}{2}=\frac{(p-m_1+2m_2-2)\gamma}{\gamma-2}$, we infer from (\ref{48})
\begin{align}\label{49}
  \frac{1}{p}\frac{d}{dt}\int_{\Omega} u^p
  \leq& -\xi \int_{\Omega} u^{p+m_1}-b_1 \int_{\Omega} u^{p+\alpha}+(\varepsilon+C_{\varepsilon})  \int_{\Omega} u^{p-m_1+2m_2-1} +C_{47}\nonumber\\
  \leq &- \int_{\Omega} u^p+C_{48},
\end{align}
where the second inequality follows from (\ref{42}).  Solving (\ref{49}) implies that $\Vert u(\cdot,t)\Vert_{L^p}$ is uniformly bounded in time for each $p\geq2$ and this completes the proof.
\end{proof}

\begin{remark}\label{remarkl}
If $m_1 \geq \alpha$, then Theorem \ref{theorem41} also holds if (\ref{42}) is relaxed to $2m_2-m_1\leq\max\{\alpha,m_1\}+1$ or equivalently $2m_2-2m_1\leq 1$.  Indeed, in this case we can choose $\xi>2(\varepsilon +C(\varepsilon))$ and therefore (\ref{49}) implies that
\[ \frac{1}{p}\frac{d}{dt}\int_{\Omega} u^p
  \leq-\frac{\xi}{2} \int_{\Omega} u^{p+m_1}+C_{48},\]
from which the boundedness of $\int_\Omega u^p(x,t)$ follows.  Similarly one can show that Theorem \ref{theorem41} holds for $2m_2-m_1\leq\max\{\alpha,m_1\}+1$ when $m_1<\alpha$ and $b_1$ is large.
\end{remark}
By a different approach we prove the following results.
\begin{theorem}\label{theorem42}
Let $\Omega$ be a bounded domain in $\mathbb{R}^N$, $N\geq1$, with piecewise smooth boundary.  Suppose that
\begin{align}\label{410}
m_2<\max\{\alpha,m_1\}
\end{align}
then the nonnegative solution $(u,v)$ to (\ref{41}) is classical, global and bounded in $\Omega\times (0,\infty)$.
\end{theorem}

\begin{proof}
We begin with (\ref{43}) and estimate the second term differently.  For each $p>2$ we denote
\[\Phi_p(u)=\int_0^u \phi(s)s^{p-2}ds,\]
then thanks to $\phi(s)\leq M_2s^{m_2}$ we have
\[\Phi_p(u)\leq M_2\int_0^u s^{p+m_2-2}ds=\frac{M_2}{p+m_2-1}u^{p+m_2-1}.\]
Therefore we have from the integration by parts and the second equation in (\ref{41})
\begin{align}\label{411}
  & -(p-1)\int_{\Omega} \chi \phi(u)u^{p-2}\nabla u\nabla v    \nonumber\\
 =& -(p-1) \chi\int_{\Omega} \nabla \Phi_p(u)\nabla v= (p-1) \chi\int_{\Omega}   \Phi_p(u)\Delta v  \nonumber\\
 =&-(p-1) \chi\int_{\Omega}   \Phi_p(u)(a_2-b_2u-c_2v)v \nonumber\\
 =&b_2(p-1) \chi\int_{\Omega}   \Phi_p(u) uv+(p-1) \chi\int_{\Omega} \Phi_p(u)(c_2v-a_2)v \nonumber\\
\leq&\frac{b_2(p-1)\chi M_2\Vert v \Vert_{L^\infty}}{p+m_2-1} \int_{\Omega} u^{p+m_2}+\frac{b_2(p-1)\chi M_2\Vert (c_2v-a_2)v  \Vert_{L^\infty}}{p+m_2-1}\int_{\Omega}  u^{p+m_2-1} \nonumber\\
\leq & C_{410} \int_{\Omega} u^{p+m_2}+C_{411},
 \end{align}
where $C_{410}$ and $C_{411}$ are positive constants.

Collecting (\ref{44}) and (\ref{411}) we have from $m_2<\max\{ \alpha,m_1\}$ that
\begin{align}\label{412}
  \frac{1}{p}\frac{d}{dt}\int_{\Omega} u^p\leq& -\xi \int_{\Omega} u^{p+m_1}-b_1 \int_{\Omega} u^{p+\alpha}+C_{410} \int_{\Omega} u^{p+m_2}+C_{412}\nonumber\\
  \leq &-\int_\Omega u^p+C_{413},
\end{align}
and this implies the boundedness of $\int_\Omega u^p$ for each $p>2$.  The rest of the proof is the same as that of Theorem \ref{theorem41}
\end{proof}

\begin{remark}\label{remark2}
Similar as Remark \ref{remarkl}, one can show that (\ref{410}) can be relaxed to $m_2\leq\max\{\alpha,m_1\}$ if $m_1>\alpha_1$ or $b_1$ is large.
\end{remark}

\subsection{Parabolic--elliptic system with attraction}

Finally, we establish the global existence and boundedness of the following parabolic--elliptic system with attraction
\begin{equation}\label{415}
\left\{
\begin{array}{ll}
u_t=\nabla \cdot (D_1(u) \nabla u-\chi \phi(u) \nabla v)+(a_1-b_1u^{\alpha}-c_1v)u,&x \in \Omega,t>0, \\
0=D_2\Delta v+(a_2-b_2u-c_2v)v,&x \in \Omega,t>0, \\
\frac{\partial u}{\partial \textbf{n}}=\frac{\partial v}{\partial \textbf{n}}=0,&x\in\partial \Omega,t>0,\\
u(x,0)=u_0(x)\geq 0, &x\in \Omega.
\end{array}
\right.
\end{equation}
Here we assume that $D_1(u)\geq M_1(1+u)^{m_1}$ as in (\ref{11}) and in contrast to (\ref{13}) changes to $\phi(u)\geq M_2u^{m_2}$.  We prove global existence and boundedness for (\ref{415}) for any $m_i\in\mathbb R^+$ and $\alpha\in\mathbb R$.  The last main result of this paper is the following theorem.
\begin{theorem}\label{theorem43}
Let $\Omega$ be a bounded domain in $\mathbb{R}^N$, $N\geq1$, with piecewise smooth boundary.  Suppose that for some positive constants $M_i>0$, $D_1(u)\geq M_1(1+u)^{m_1}$ and $\phi(u)\geq M_2u^{m_2}$, with $\max\{m_1,m_2,\alpha\}\geq 0$.  Then the nonnegative solution $(u,v)$ to (\ref{415}) is classical and uniformly bounded in $\Omega\times (0,\infty)$.
\end{theorem}
\begin{proof}
By the same arguments for (\ref{43}) we test the $u$-equation in (\ref{415}) by $u^{p-1}$ and integrate it over $\Omega$ by parts to obtain
\begin{align}\label{416}
\frac{1}{p}\frac{d}{dt}\int_{\Omega} u^p=&-(p-1)\int_{\Omega} D_1(u)u^{p-2}|\nabla u|^2+(p-1)\int_{\Omega} \chi \phi(u)u^{p-2}\nabla u\nabla v \nonumber\\
&+\int_{\Omega} u^p(a_1-b_1u^\alpha-c_1v).
\end{align}
Similar as above, we denote
\[\tilde \Phi_p(u):=\int_0^u \phi(s)s^{p-2}ds.\]
Then we can easily show that $\tilde \Phi_p(u)\geq \frac{M_2}{p+m_2-1}u^{p+m_2-1}$ and we can derive as in (\ref{411}) that
\begin{align}\label{417}
  & (p-1)\int_{\Omega}\chi  \phi(u)u^{p-2}\nabla u\nabla v    \nonumber\\
= & -b_2(p-1) \chi\int_{\Omega}\tilde\Phi_p(u) uv+(p-1) \chi\int_{\Omega} \tilde\Phi_p(u)(c_2v-a_2)v \nonumber\\
\leq & -C_{414}\int_{\Omega}  u^{p+m_2}+C_{415},
\end{align}
where $C_{414}$ and $C_{415}$ are positive constants.  Collecting (\ref{417}) and (\ref{44}), we have from (\ref{416})
\begin{align}\label{418}
\frac{1}{p}\frac{d}{dt}\int_{\Omega} u^p
\leq&-\xi \int_\Omega u^{p+m_1}-b_1\int_{\Omega} u^{p+\alpha}-C_{415} \int_{\Omega} u^{p+m_2} +C_{415}\nonumber\\
\leq&-  \int_\Omega u^p+C_{416},
\end{align}
where the last inequality follows from Young's inequality and the assumption that $\max\{m_1,m_2,\alpha\}\geq0$.  Solving (\ref{418}) gives rise to the boundedness of $\int_\Omega u^p$ for any $p>2$ hence the global existence and boundedness follow.
\end{proof}
According to Theorem \ref{theorem43}, only one of $m_1$, $m_2$ and $\alpha$ is needs to be nonnative to guarantee the global existence and boundedness of (\ref{415})in contrast to Theorem \ref{theorem41} and Theorem \ref{theorem42}.  Apparently, this is due to the effect of population attraction.   It is necessary to point out that for chemotaxis model, it is well known that chemo--attraction destabilizes the system and supports the occurrence of blowups, while chemo--repulsion tends to prevent blowups.  However, for Lotka--Volterra competition models, attraction prevents blowups while repulsion, though it is not completely understood, tends to support blowups.  See \cite{LeN2,WGY} for instance.  We surmise that the same conclusions hold true for the fully parabolic system (\ref{11}), though a completely different approach is needed for this purpose.

\begin{remark}\label{remark3}
Consider the following system
\begin{equation}\label{419}
\left\{
\begin{array}{ll}
u_t=\nabla \cdot (D_1(u) \nabla u+\chi \phi(u) \nabla v)+(a_1-b_1u-c_1v)u,&x \in \Omega,t>0, \\
v_t=D_2\Delta v+(a_2-b_2u-c_2v)v,&x \in \Omega,t>0, \\
\frac{\partial u}{\partial \textbf{n}}=\frac{\partial v}{\partial \textbf{n}}=0,&x\in\partial \Omega,t>0,\\
u(x,0)=u_0(x)\geq 0, &x\in \Omega.
\end{array}
\right.
\end{equation}
When $\chi>0$, it is known \cite{WGY,WZ} that (\ref{419}) admits nontrivial spatial patterns when $\chi$ is large.  It is interesting to investigate the attraction model with $\chi<0$.
\end{remark}

\end{document}